\documentclass[12pt]{amsart}
\usepackage{amsmath,amsthm,amsfonts,amssymb}%\usepackage{paralist,nicefrac}

\newcommand{\bbC}{\mathbb{C}}
\newcommand{\bbZ}{\mathbb{Z}}
\newcommand{\bbR}{\mathbb{R}}
\newcommand{\bbA}{\mathbb{A}}
\newcommand{\bbF}{\mathbb{F}}
\newcommand{\bbL}{\mathbb{L}}

\newcommand{\Spec}{\operatorname{Spec}\,}

\def\Groth#1{K_0(\mathrm{Var}_{#1})}

\newtheorem{thm}{Theorem}

\newtheorem{cor}[thm]{Corollary}

\newtheorem{prop}[thm]{Proposition}

\newtheorem{lem}[thm]{Lemma}

\newtheorem*{claim*}{Claim}
\begin{document}

\title{Positive Motivic Measures are Counting Measures}

\author{Jordan S. Ellenberg}
\address{Department of Mathematics,
University of Wisconsin,
480 Lincoln Dr.,
Madison, WI,
53706, 
U.S.A}

\author{Michael Larsen}
\address{Michael Larsen,
Department of Mathematics,
Indiana University,
Bloomington, IN
47405,
U.S.A.}

\maketitle

Let $K$ be a field.  By a \emph{$K$-variety}, we mean a geometrically reduced, separated scheme of finite type over $K$.  Let $\Groth K$ denote the Grothendieck group of $K$, i.e., the free abelian group generated by isomorphism classes $[V]$ of $K$-varieties, with the scissors relations 
$[V] = [W] - [V\setminus W]$ whenever $W$ is a closed $K$-subvariety of $V$.  There is a unique product on $\Groth K$ characterized by the relation
$$[V]\cdot [W] = [V\times W],$$
where $\times$ denotes fiber product over $\Spec K$. 
For every extension $L$ of $K$, extension of scalars gives a natural ring homomorphism $\Groth K\to\Groth L$.  The map $K\mapsto \Groth K$ can be regarded as a functor from fields to commutative rings.  Throughout the paper, we follow the usual convention of writing 
$\bbL$ for $[\bbA^1_K]$.

A ring homomorphism from $\Groth K$ to a field $F$ is called a \emph{motivic measure}.
See, e.g., \cite{Hales,Looijenga} for general information on motivic measures.
If $K$ is a finite field, the map $[V]\mapsto |V(K)|$ extends to a homomorphism 
$\mu_K\colon \Groth K\to \bbZ$, 
and therefore to an $F$-valued measure for any field $F$.  More generally, if $L$ is an extension of $K$
which is also a finite field, the composition of $\mu_L$ with the natural map $\Groth K\to \Groth L$ gives for each $F$ a motivic measure.  We will call all such measures \emph{counting measures}.

In this paper, we consider \emph{positive} motivic measures, by which we mean
$\bbR$-valued measures $\mu$ such that $\mu([V])\ge 0$ for all $K$-varieties $V$.  Our main result is the following:

\begin{thm}
\label{main}
Every positive motivic measure is a counting measure.  In other words, if $\mu\colon \Groth K\to \bbR$ is positive, there exists a finite field $L$ containing $K$ such that $\mu([V]) = |V(L)|$ for all $K$-varieties 
$V$.
\end{thm}

Of course, for other choices of $F$ there may still be motivic measures such that $\mu([V])$ lies in some interesting semiring of $F$ for all $K$-varieties $V$.  For example, if $F$ is $\bbC(u,v)$ and $K=\bbC$, the measure sending $V$ to its Hodge-Deligne polynomial takes values in the semiring of polynomials in $u,v$ whose term of highest total degree is a positive multiple of a power of $uv$.  

We begin with a direct proof of the following obvious corollary of Theorem~\ref{main}.

\begin{prop}
If $K$ is infinite, there are no positive motivic measures on $\Groth K$.
\end{prop}

\begin{proof}
Let $\mu$ be such a measure.  For any finite subset $S$ of $K$, which we regard as a zero-dimensional subvariety of $\bbA^1$,
$$0\le \mu(\bbA^1\setminus  S) = \mu(\bbL) - |S|.$$
Thus, $\mu(\bbL) \ge |S|$ for all subsets $S$ of $K$, which proves the proposition.
\end{proof}

For the remainder of the paper we may and do assume that $K$ is finite, of cardinality $q$.
We write $\bbF_{q^n}$ for the degree $n$ extension of $K$.

\begin{prop}
Let $\Omega^n$ denote the variety obtained from $\bbA^n$ by removing all proper affine-linear subspaces defined over $\bbF_q$.  Then
$$[\Omega^n] = (\bbL-q)(\bbL-q^2)\cdots(\bbL-q^n).$$
\end{prop}

\begin{proof}
For any $\bbF_q$-rational affine-linear subspace $A$ of $\bbA^n$, let $A^\circ$ denote the open subvariety of $A$ which is the complement of all proper $\bbF_q$-rational affine-linear subspaces of $A$.  Then $[A^\circ] = [\Omega^{\dim A}]$, and one can write recursively
$$[\Omega^n] = \bbL^n - \sum_{i=1}^{n-1} a_{n,i}[\Omega^i],$$
where $a_{n,i}$ is the number of $\bbF_q$-rational $i$-dimensional affine linear subspaces of
$\bbA^n$.  Thus, $[\Omega^n]$ can be expressed as $P_n(\bbL)$, where $P_n\in\bbZ[x]$ is monic and of degree $n$.  It suffices to prove that $q^d$ is a root of $P_n(x)$ for all integers $d\in\{1,2,\ldots,n\}$.

For any $d$ in this range $\Omega^n(\bbF_{q^d})$ is empty.  Indeed, if $x\in \bbA^n(\bbF_{q^d})$, then the $n$ coordinates of $x$ together with $1$ cannot be linearly independent over $\bbF_q$, which implies that $x$ lies in a proper $\bbF_q$-rational affine-linear subspace of $\bbA^n$.  Thus,
$$0 = \mu_{\bbF_{q^d}}(\Omega^n) = P_n(q^d).$$
\end{proof}

\begin{cor}
If $\mu$ is a positive measure on $\Groth {\bbF_q}$, there exists a positive integer $n$ such that
$\mu(\bbL) = q^n$.
\end{cor}

\begin{proof}
If $q^{n-1} < \mu(\bbL) < q^n$ for some integer $n$, then $\mu(\Omega^n) < 0$, contrary to positivity.
\end{proof}

Our goal is then to prove that $\mu(\bbL) = q^n$ implies $\mu = \mu_{\bbF_{q^n}}$.  We prove first that
these measures coincide for varieties of the form $\Spec \bbF_{q^d}$ and deduce that they coincide for all affine varieties.  As $\Groth {\bbF_q}$ is generated by the classes of affine varieties, this implies the theorem.

\begin{lem}
\label{fiber-product}
Let $\mu$ be a real-valued motivic measure of $\Groth {\bbF_q}$ and $m$ a positive integer.  Then 
$$\mu(\Spec \bbF_{q^m})\in\{0,m\}.$$
Moreover, if $\Spec\bbF_{q^m}$ has measure $m$ , then $\Spec\bbF_{q^d}$ has measure $d$ whenever $d$ divides $m$.
\end{lem}

\begin{proof}
As
$$\bbF_{q^m}\otimes_{\bbF_q} \bbF_{q^m} = \bbF_{q^m}^m,$$
the class of $\Spec \bbF_{q^m}$ satisfies $x^2 = mx$.
If $d$ divides $m$,
$$\bbF_{q^d}\otimes _{\bbF_q} \bbF_{q^m} = \bbF_{q^m}^d,$$
so $\mu(\Spec \bbF_{q^m})=m$ implies $\mu(\Spec \bbF_{q^d})=d$.
\end{proof}

Of course, 
$$\mu_{\bbF_n}(\Spec \bbF_{q^m}) = 
\begin{cases}
m&\text{if $m|n$}\\
0&\text{otherwise.}
\end{cases}$$

We would like to prove the same thing for the values of $\mu(\Spec\bbF_{q^m})$.   We begin with the following proposition.

\begin{prop}
If $\mu(\bbL)=q^n$ and $\mu(\Spec(\bbF_{q^k})) = k$ for some $k\ge n$, then
\begin{equation}
\label{zero-dim}
\mu(\Spec \bbF_{q^m}) = 
\begin{cases}
m&\text{if $m|n$}\\
0&\text{otherwise.}
\end{cases}
\end{equation}
\end{prop}

For any integer $k$, we denote by $X_k$ the complement in $\bbA^1$ of the set of all points 
with residue field contained in $\bbF_{q^k}$.

\begin{proof}
By Lemma~\ref{fiber-product}, $\mu(\Spec \bbF_{q^d}) = d$ when $d$ divides $k$.
Choose an $m$ not dividing $k$, and let $Y_{k,m}$ denote the complement in $X_k$ of the set of points with residue field 
$\bbF_{q^m}$.
Then
$$\mu([Y_{k,m}]) = \mu(\bbL) - \sum_{d\mid k} c_d d - c_m \mu(\Spec\bbF_{q^m}),$$
where $c_i$ is the number of points in $\bbA^1$ with residue field $\bbF_{q^i}$.  From the positivity of $\mu([Y_{k,m}])$ and the fact that
$$0 = \mu_{\bbF_{q^k}}([Y_{k,m}]) = q^k - \sum_{d\mid k} c_d d$$
we see that $\mu(\bbL) - q^k = q^n - q^k$ must be nonnegative, which is to say $k=n$, and that $\mu(\Spec\bbF_{q^m}) = 0$.

%{\bf JE:  I shortened this because I thought it was conceptually cleaner without the computation, change back if you disagree or if I screwed it up somehow.}

% or, equivalently, the number of Galois-orbits in 
%
%$$\bbF_{q^i} - \bigcup_{j\mid i} \bbF_{q^j}.$$
%
%This is easily seen to be
%
%$$\frac{\sum_{j\mid i} \mu(i/j)q^j}i.$$
%
%Thus,
%
%$$\mu([Y_{k,m}]) = q^n - \sum_{d\mid k}\sum_{j\mid d} \mu(d/j)q^j - c_m m = (q^n - q^k)-c_m m < 0,$$
%
%contrary to positivity.  Thus no such $m$ exists.  As
%
%$$0\le \mu([X_k]) = q^n - q^k,$$
%
%we have $k=n$.
\end{proof}

\begin{prop}
If $\mu(\bbL) = q^n$, then $\mu(\Spec \bbF_{q^n}) = n$.
\end{prop}

\begin{proof}
%{\bf JE:  changed some stuff around the inequalities here just because I was concerned about boundary cases, etc. -- let me know if this is OK}

The assertion is clear for $n=1$, so we assume $n > 1$.  Let $c_i$  denote the number of points in $\bbA^1$ with residue field $\bbF_{q^i}$. Thus $ic_i \le q^i - 1$ for all $i > 1$.  If $\mu(\Spec \bbF_{q^n}) = 0$, then
$\mu(\Spec(\bbF_{q^i})) = 0$ for all $i\ge n$, so for all $k > 0$ we have 
$$\mu([X_k]) \ge q^n - q - \sum_{i=2}^{n-1} (q^i - 1) \ge 2.$$
Now we consider all curves in $\bbA^2$ of the form $y = P(x)$ where $P(x)\in \bbF_q[x]$
has degree $\le 2n$.  The total number of such curves is greater than $q^{2n}$, and for any intersection
point $(\alpha,\beta)$ of any two distinct curves of this family, $\alpha$ satisfies a polynomial equation of degree $\le 2n$ over $\bbF_q$.  Therefore, the open curves
$$C_P := \{(x,P(x))\mid x\not\in \bbF_{q^{(2n)!}}\},$$
indexed by polynomials $P$ of degree $\le 2n$, each isomorphic to $X_{(2n)!}$, are mutually disjoint.  If
$C$ denotes the closure of the union of the $C_P$ in $\bbA^2$, it follows that
$$\mu([C]) > q^{2n}\mu([X_{(2n)!}]) > q^{2n},$$
so $\mu([\bbA^2\setminus C]) < 0$, which is absurd.
\end{proof}

Together, the two preceding propositions imply equation
(\ref{zero-dim}).

We can now prove Theorem~\ref{main}.  We assume $\mu(\bbL)=q^n$.  It suffices to check that $\mu([V]) = |V(\bbF_{q^n})|$ for all affine $\bbF_q$-varieties $V$.

Each closed point of $V$ with residue field $\bbF_{q^d}$ corresponds to a $d$-element Galois orbit in $V(\bbF_{q^d})$.  If $d$ divides $n$, it gives a $d$-element subset of $V(\bbF_{q^n})$ and the subsets arising from different closed points are mutually disjoint. Since $V(\bbF_{q^n})$ is the union of all these subsets, and  $\mu(\Spec\bbF_{q^d}) = d$, we have
\begin{equation}
\label{V-ineq}
\mu([V]) \ge |V(\bbF_{q^n})|
\end{equation}
for each $\bbF_q$-variety $V$.  However, embedding $V$ as a closed subvariety of $\bbA^m$ for some $m$, the complement $W = \bbA^m\setminus V$ is again a variety, so
\begin{equation}
\label{W-ineq}
\mu([W]) \ge |W(\bbF_{q^n})|.
\end{equation}
As
$$q^{mn} = \mu([\bbA^m]) = \mu([V]) + \mu([W]) \ge |V(\bbF_{q^n})|+|W(\bbF_{q^n})| = 
|\bbA^m(\bbF_{q^n})| = q^{mn},$$
we must have equality in (\ref{V-ineq}) and (\ref{W-ineq}).

\end{document}